\documentclass{article}

\usepackage{amssymb}
\usepackage{amsfonts}
\usepackage{amsmath}
\usepackage{amsthm}
\usepackage[small,nohug,heads=vee]{diagrams}
\usepackage{url}

\newcommand{\bd}{\begin{description}}
\newcommand{\ed}{\end{description}}

\newtheorem{theorem}{Theorem}[section]
\newtheorem{lemma}[theorem]{Lemma}

\newtheorem{remark}[theorem]{Remark}
\newtheorem{corollary}[theorem]{Corollary}
\newtheorem{algorithm}[theorem]{Algorithm}

\newtheorem{example}[theorem]{Example}
\newtheorem{definition}[theorem]{Definition}

\newcommand{\bi}{\begin{itemize}}
\newcommand{\ei}{\end{itemize}}

\def\GF{\mathbb{F}}
\def\Alt{{\rm Alt}}
\def\Sym{{\rm Sym}}
\def\2G2{\ensuremath{^2{\rm G}_2}}
\def\sl{{\rm{SL}}}
\def\gl{{\rm{GL}}}

\def\su{{\rm{SU}}}
\def\psl{{\rm{PSL}}}

\def\so{{\rm{SO}}}

\def\ppom{{\rm{(P) \Omega }}}
\def\om{{\rm{\Omega }}}
\def\sp{{\rm{Sp}}}
\def\psp{{\rm{PSp}}}
\def\ppsl{ ( {\rm{P}} ) {\rm{SL}}}
\def\ppsp{ ( {\rm{P}} ) {\rm{Sp}}}
\def\ppsu{ ( {\rm{P}} ) {\rm{SU}}}

\def\ppone{p \equiv 1 \bmod 4}
\def\pmone{p \equiv -1 \bmod 4}
\def\qpone{q \equiv 1  \bmod 4}
\def\qmone{q \equiv -1 \bmod 4}

\title{Steinberg presentations of black box classical groups in small characteristics}

\author{Alexandre Borovik\footnote{School of Mathematics, University of Manchester, UK; alexandre.borovik@gmail.com} and \c{S}\"{u}kr\"{u} Yal\c{c}\i nkaya\footnote{Bilgi University, Istanbul, Turkey; sukru.yalcinkaya@gmail.com}}

\begin{document}

\maketitle

\begin{abstract}
The main component of (constructive) recognition algorithms for black box groups of Lie type in computational group theory is the construction of unipotent elements. In the existing algorithms unipotent elements are found by random search and therefore the running time of these algorithms is polynomial in the underlying field size $q$ which makes them unfeasible for most practical applications  \cite{guralnick01.169}. Meanwhile, the input size of recognition algorithms involves only $\log q$. The present paper introduces a new approach to construction of unipotent elements in which the running time of the algorithm is quadratic in characteristic $p$  of the underlying field and is polynomial in $\log q$; for small values of $p$ (which make a vast and practically important class of problems), the complexity of these algorithms is polynomial in the input size.

For $\psl_2(q)$, $\qpone$, we present a Monte-Carlo algorithm which constructs a root subgroup $U$, the maximal torus $T$ normalizing $U$ and a Weyl group element $w$ which conjugates $U$ to its opposite. Moreover, we extend this result and construct Steinberg generators for the black box untwisted classical groups defined over a field of odd size $q=p^k$ where $\qpone$. Our algorithms run in time quadratic in characteristic $p$ of the underlying field and polynomial in $\log q$ and the Lie rank $n$ of the group.

The case $\qmone$ requires the use of additional tools and is treated separately in our next paper \cite{suko12B}. Further, and much stronger results can be found in \cite{suko12E,suko12F}.
\end{abstract}

\begin{footnotesize}
\tableofcontents
\end{footnotesize}

\section{Introduction and the principal results}

\subsection{Black box groups}

The purpose of the present paper is to introduce an efficient algorithm which constructs   the so-called Steinberg generators of black box classical groups in small odd characteristics; it will be used in subsequent papers \cite{suko12E,suko12F} for recovery of the structure of these groups.

Black box groups were introduced by Babai and Szemeredi \cite{babai84.229} as an idealized setting for randomized algorithms for solving permutation and matrix group problems in computational group theory.

A black box group $X$ is a black box (or an oracle, or a device, or an algorithm) operating with $0$-$1$ strings of bounded length which encrypt (not necessarily in a unique way) elements of some finite group $G$ (in various classes of black box problems the isomorphism type of $G$ could be known in advance or unknown). The functionality of the black box is specified by the following axioms: the black box

\begin{itemize}
\item[\textbf{BB1}] produces strings encrypting random elements from $G$;
\item[\textbf{BB2}]computes a string encrypting the product of two group elements given by strings or a string encrypting the inverse of an element given by a string; and
\item[\textbf{BB3}] compares whether two strings encrypt the same element in $G$---therefore we have a canonical map (not necessarily easily computable in practice) $\pi: X \to G$.
\end{itemize}

We shall say in this situation that $X$ is a \emph{black box over $G$} or that $X$ \emph{encrypts} $G$.

A typical example is provided by a group $G$ generated in a big matrix group $\gl_n(r^k)$ by several matrices $g_1,\dots, g_l$. The product replacement algorithm \cite{celler95.4931} produces  a sample of (almost) independent elements from a distribution on $G$ which is close to the uniform distribution (see the discussion and further development in \cite{babai00.627, babai04.215, bratus99.91, gamburd06.411, lubotzky01.347, pak00.476, pak01.301, pak.01.476, pak02.1891}). We can, of course, multiply, invert, compare matrices. Therefore the computer routines for these operations together with the sampling of the product replacement algorithm run on the  tuple of generators $(g_1,\dots, g_l)$ can be viewed as a black box $X$ encrypting the group $G$. The group $G$ could be unknown---in which case we are interested in its isomorphism type---or it could be known, as it happens in a variety of other black box problems. For example, if we already know that $G$ is isomorphic to, say, $\sl_{2m}(r^s)$, we may wish to construct in $G$ subgroups $H_1\cong \sp_{2m}(r^s)$, $H_2\cong \sl_{2m}(r)$ and $H_3 \cong \sp_{2m}(r)$ in such a way that $H_1\cap H_2 = H_3$. (This problem is actually solved in one of the  subsequent papers in this series \cite{suko12B}.) In our set-up, this means that we wish to construct black boxes $Y_i$, $i=1,2,3$, over $H_i$ and embeddings $Y_i \to X$. This formalism is further developed in \cite{suko12F,suko12D}.

Notice that even in routine examples the number of elements of a matrix group $G$ could be astronomical, thus making many natural questions about the black box $X$ over $G$---for example, finding the isomorphism type or the order of $G$---inaccessible for all known deterministic methods. Even when $G$ is cyclic and thus is characterized by its order, existing approaches to finding multiplicative orders of matrices over finite fields are conditional and involve oracles either for the discrete logarithm problem in finite fields or for prime factorization  of integers.

Nevertheless black box problems for matrix groups have a feature which makes them more accessible:

\begin{itemize}
\item[\textbf{BB4}] We are given a \emph{global exponent} of $X$, that is, a natural number $E$ such that it is expected that $x^E = 1$ for all elements $x \in X$ while computation of $x^E$ is computationally feasible.
\end{itemize}

Usually, for a black box group $X$ arising from a subgroup in the ambient group $\gl_n(r^k)$, the exponent of $\gl_n(r^k)$ can be taken for a global exponent of $X$.

Abusing terminology, in this paper we shall frequently identify the black box $X$ and the group $G$ encrypted by $X$ (as we have already done in formulation of Axiom BB4); this is relatively safe in simpler black box problems about matrix groups over finite fields. However, more sophisticated algorithms which we shall discuss in subsequent papers will require a certain level of hygiene which will make  identification of black box groups with the groups which they encrypt inconvenient.

\begin{quote}
\emph{In this paper, we assume that all our black box groups satisfy assumptions BB1--BB4.}
\end{quote}

We emphasise that we do not assume that black box groups under consideration in this paper are given as subgroups of ambient matrix groups; thus our approach is wider than that of the computational matrix group project \cite{leedham-green01.229}. This makes us to be a more careful with basic terminology. In particular, given two black boxes $X,Y$ encrypting groups $G,H$, correspondingly, we say that a map $\alpha$ which assigns strings from $X$ to strings from $Y$ is a morphism of black box groups, if there is a homomorphism $\beta:G \to H$ such that the following diagram  is commutative:
\begin{diagram}
X &\rTo^{\alpha} &Y\\
\dDotsto_{\pi_X} & &\dDotsto_{\pi_Y}\\
G &\rTo^{\beta} & H
\end{diagram}
(here $\pi_X$ and $\pi_Y$ are canonical projections of $X$ and $Y$ onto $G$ and $H$, correspondingly).

\subsection{Black box group problems}

We shall outline an hierarchy of typical black box group problems.

\begin{description}
\item[\textbf{Verification Problem.}] Is the unknown group encrypted by a black box group $X$ isomorphic to the given group $G$ (``target group'')?

\item[\textbf{Recognition Problem.}] Determine the isomorphism class of the group encrypted by $X$.
\end{description}

The Verification  Problem is rarely discussed in the literature on black box groups on its own but frequently arises as a sub-problem within more complicated Recognition Problems. The two problems have dramatically different complexity. For example, the celebrated Miller-Rabin algorithm \cite{rabin80.128} for testing primality of the given odd natural number $n$ in nothing else but a  black box algorithm for solving the verification problem for the multiplicative group   $\mathbb{Z}/n\mathbb{Z}^*$ of residues modulo $n$ (given by a simple black box: take your favorite random numbers generator and generate random integers between $1$ and $n$) and the cyclic group $\mathbb{Z}/(n-1)\mathbb{Z}$ of order $n-1$ as the target group. On the other hand, if $n=pq$ is the product of primes $p$ and $q$, the recognition problem for the same black box group means finding the direct product decomposition
\[
\mathbb{Z}/n\mathbb{Z}^* \cong \mathbb{Z}/(p-1)\mathbb{Z} \oplus \mathbb{Z}/(q-1)\mathbb{Z}
\]
which is equivalent to factorization of $n$ into product of primes.

The next step after finding the isomorphism type of the black box group $X$ is

\begin{description}
\item[Constructive Recognition.] Suppose that a black box group $X$ encrypts a concrete and explicitly given group $G$. Rewording a definition given in \cite{brooksbank08.885},
    \begin{quote}
         The goal of a constructive recognition algorithm is
to construct an effective isomorphism $\Psi: G \longrightarrow X$. That is, given $g\in G$, there is an efficient procedure to construct a string $\Psi(g)$ representing $g$  in $X$ and given a string $x$  produced by $X$, there is an efficient procedure to construct the element $\Psi^{-1}(x) \in G$ represented by $X$.
    \end{quote}
\end{description}

However, there are still no really efficient constructive recognition algorithms for black box groups $X$ of (known) Lie type over a finite field of large order $q=p^k$. The first computational obstacles for known algorithms
\cite{brooksbank01.95,brooksbank01.79,brooksbank03.162,brooksbank06.256,brooksbank08.885,celler98.11,conder06.1203,leedham-green09.833} are the need to construct unipotent elements in black box groups, \cite{brooksbank01.95,brooksbank01.79,brooksbank03.162,brooksbank06.256,brooksbank08.885,celler98.11} or to solve discrete logarithm problem for matrix groups \cite{conder01.113,conder06.1203,leedham-green09.833}.

Unfortunately,
the proportion of the unipotent elements in $X$ is $O(1/q)$ \cite{guralnick01.169}. Moreover the probability that the order of a random element is divisible by $p$ is also $O(1/q)$, so one has to make  $O(q)$ (that is, \emph{exponentially many}, in terms of the input length $O(\log q)$ of the black boxes and the algorithms) random selections of elements in a given group to construct a unipotent element. However, this brute force approach is still working for small values of $q$, and  Kantor and Seress \cite{kantor01.168} used it to develop an algorithm for recognition of black box classical groups. Later the algorithms of \cite{kantor01.168} were upgraded to polynomial time constructive recognition algorithms \cite{brooksbank03.162, brooksbank01.95, brooksbank06.256, brooksbank08.885} by assuming the availability of additional \emph{oracles}:
 \begin{itemize}
 \item the \emph{discrete logarithm oracle} in $\mathbb{F}_q^*$, and
  \item the  \emph{$\sl_2(q)$-oracle}.
   \end{itemize}

   The latter is a procedure for the constructive recognition of $\sl_2(q)$; see  discussion in \cite[Section~3]{brooksbank08.885}.

\begin{description}
\item[Structure recovery.] Suppose that a black box group $X$ encrypts a concrete and explicitly given group $G$. A weaker, but frequently feasible and very useful version of constructive recognition is what we call \emph{structure recovery}: construction of a probabilistic polynomial time morphism $$\Psi: G \longrightarrow X.$$ That is, given $g\in G$, there is an efficient procedure to construct a string $\Psi(g)$ representing $g$  in $X$---but we do not require that the map $\Psi$ can be efficiently reversed.
\end{description}

Structure recovery of  black box groups encrypting Chevalley groups in odd characteristic is the principal aim of papers \cite{suko12E,suko12F, suko12B,suko12C}, the present paper prepares some scaffoldings for this work. A more detailed discussion of methodological issues  could be found in \cite{suko12F,suko12D}.

\subsection{Results of the paper}

This paper is the first in the series of works \cite{suko12E,suko12F,suko12D,suko12B,suko12C} directed at development of polynomial time methods of computing in black box groups without seeking help from any kind of oracles.

As we have already mentioned, for sake of compactness of exposition in this paper we do not make a notational distinction between a black box and the group encrypted by it. However, in view of the use of results of this paper in subsequent work we carefully underly this distinction in the statements of results.

As the first step, we find unipotent elements in black box groups of Lie type of small odd characteristic.

\begin{theorem}\label{cons:uni}
Let $X$ be a black box group encrypting a quasi-simple group of Lie type of odd characteristic $p$ over a field of size $q=p^k>3$. If\/ $p\neq 5$ or $7$, then there exists a Monte-Carlo algorithm which constructs a string representing a unipotent element. This algorithm works in time polynomial in the Lie rank $n$ of\/ $X$ and\/ $\log q$, and is quadratic in $p$.

The same result holds if\/ $p=5$ or\/ $7$ and $k$ has a small divisor $l$, with the algorithm running in time polynomial in $n$ and\/ $\log q$, and quadratic in $p^l$.
\end{theorem}

Then we extend this result to present an algorithm that constructs the Steinberg generators of the classical groups. The groups $\ppsl_2(q)$ can be viewed as the starting point of recursion and we first present an algorithm for $\ppsl_2(q)$.

We need to recall the notion of Steinberg generators of $\ppsl_2(q)$ as introduced by Steinberg \cite[Theorem 8]{steinberg1968}.

Let $G=\sl_2(q)$. Then, for $t \in \GF(q)$, set
\begin{eqnarray}\label{steinberg:sl2}
u(t)= \left[
\begin{array}{cc}
1  & t  \\
0   &  1 \\
\end{array}
\right],
\,
v(t)= \left[
\begin{array}{cc}
1  & 0  \\
t   &  1 \\
\end{array}
\right],
\,
h(t)= \left[
\begin{array}{cc}
t  & 0  \\
0   &  t^{-1} \\
\end{array}
\right],
\,
n(t)= \left[
\begin{array}{cc}
0  & t  \\
-t^{-1}   &  0 \\
\end{array}
\right]
\end{eqnarray}
where $t\neq 0$ for $h(t)$ and $n(t)$.  It is straightforward to check that
\begin{eqnarray}\label{opp:uni}
u(t)^{n(s)}=v(-s^{-2}t), \,  u(1)^{h(t)}=u(-t^2) \mbox{ and } n(1)^{h(t)}=n(t^2).
\end{eqnarray}
Moreover,
\begin{eqnarray}\label{tori:weyl}
n(t)=u(t)v(-t^{-1})u(t) \mbox{ and } h(t)=n(t)n(-1).
\end{eqnarray}
It is well-known that  \[G=\langle u(t), v(t) \mid t\in \GF(q) \rangle,\] see, for example, \cite[Lemma 6.1.1]{carter1972}. Therefore, by (\ref{opp:uni}) and (\ref{tori:weyl}), \[G= \langle u(1),h(t),n(1)\mid t \in \GF(q)^*\rangle;\]
notice that actually $G$ is generated by three elements
\[G= \langle u(1),h(t),n(1)\rangle\]
where we can take for $t$ an arbitrary primitive element of the field\/ $\mathbb{F}_q$.

We prove the following.

\begin{theorem}\label{cons:psl2B}
Let $X$ be a black box group encrypting $\ppsl_2(q)$, where $\qpone$ and $q=p^k$ for some $k\geqslant 1$. Then there is a Monte-Carlo algorithm which constructs in $X$ strings $u$, $h$, $n$   such that there exists an isomorphism\/ $$\Phi: X \longrightarrow \ppsl_2(q)$$ with
\[
\Phi(u) = \begin{bmatrix} 1&1\\ 0&1\end{bmatrix}, \Phi(h) =  \begin{bmatrix} t&0\\ 0&t^{-1}\end{bmatrix}, \Phi(n) = \begin{bmatrix} 0&1\\ -1&0\end{bmatrix},
\]
where $t$ is some primitive element of the field\/ $\mathbb{F}_q$.

The running time of the algorithm is quadratic in $p$ and polynomial in $\log q$.
\end{theorem}

Theorem~\ref{cons:psl2B} deserves some discussion and comparison with the $\sl_2(q)$-oracle as described in \cite[Section~3]{brooksbank08.885}.

Notice that $\Phi(u)$, $\Phi(h)$, and $\Phi(n)$ are some Steinberg generators of $\ppsl_2(q)$. However, not being oracles  we do not have an efficient procedure for computing  isomorphism $\Phi$, but we exhibit its inverse $\Phi^{-1}$ in our next paper \cite{suko12F}.

Still, Theorem  \ref{cons:psl2B} provides sufficient structural information about $X$ to facilitate solution of a wide range of natural problems about $\ppsl_2(q)$ and other black box groups of Lie type and odd characteristic; a detailed discussion of applications of Theorems~\ref{cons:psl2B} and its easy corollary, Theorem \ref{cons:psl2} below, can be found in  our next paper \cite{suko12B}.

\begin{theorem}\label{cons:psl2}
Let $X$ be a black box group encrypting the group $G\cong \ppsl_2(q)$, where $\qpone$ and $q=p^k$ for some $k\geqslant 1$. Then there is a Monte-Carlo algorithm which constructs a  triple $(U,T,w)$ in $X$ where $U$ is (a black box for) a root subgroup in $G$, $T$ is (a black box for) a maximal torus in $G$ normalizing $U$ and $w$ is a representative in $N_X(T)$  of a Weyl group element in $G$ which conjugates $U$ to its opposite root subgroup. The running time of the algorithm is quadratic in $p$ and polynomial in $\log q$.
\end{theorem}

The formulation of Theorem \ref{cons:psl2} reflects another aspect of our approach to black box groups: we prefer to manipulate with subgroups (defined in $X$ by their own  smaller black ``subboxes'') rather than with individual elements.

In our next result, we expand Theorem~\ref{cons:psl2} to construction of Steinberg generators in the Curtis-Tits configurations of classical groups (for discussion of the latter, see Sections~\ref{subsec:CTC-single-bond} and \ref{subsec:CTC-double-bond}). We prove the following.

\begin{theorem}\label{cons:classical}
Let $X$ be a black box classical group encrypting one of the groups $\ppsl_{n+1}(q)$, $\ppsp_{2n}(q)$, $\om_{2n+1}(q)$ or\/ $\ppom_{2n}^+(q)$, where $\qpone$ and $q>5$. Then there is an algorithm which constructs:
\begin{itemize}
\item black boxes for an extended Curtis-Tits configuration $\{K_0, K_1, \ldots, K_n\}$ of\/ $X$;
    \item black boxes for root subgroups $U_\ell < K_\ell$;
    \item a black box for a maximal torus $T$ where $T<N_X(U_\ell)$;
    \item Weyl group elements $w_\ell \in K_\ell$, where $U_\ell^{w_\ell}$ is the opposite root subgroup of $U_\ell$ in $K_\ell$  for all $\ell=0,1,\ldots,n$.
\end{itemize}
The running time of the algorithm is quadratic in the characteristic $p$ of the underlying field, and is polynomial in the Lie rank $n$ of $X$ and $\log q$.
\end{theorem}

The two families of classical groups,  $\ppsu_n(q)$ and $\ppom_{2n}^-(q)$, are not covered by Theorem~\ref{cons:classical}. They are twisted Chevalley groups whose Curtis-Tits presentations are more complicated than these of Chevalley groups, see \cite[Section 2.4]{gorenstein1998}, and work within these groups requires additional technical tools. However, we know how to develop algorithms for $\ppsu_n(q)$ and $\ppom_{2n}^-(q)$ similar to those described in this paper, they  will be published elsewhere. The corresponding algorithms for exceptional groups will be presented in our next paper \cite{suko12C}.

Theorem~\ref{cons:classical} will be used in the subsequent paper \cite{suko12E} to prove a more precise result:
\begin{theorem}\label{recovery-classical}
Let $X$ be a black box classical group encrypting one of the groups $G\simeq \ppsl_{n+1}(q)$, $\ppsp_{2n}(q)$, $\om_{2n+1}(q)$ or\/ $\ppom_{2n}^+(q)$, where $\qpone$ and $q>5$. Then there is a Monte-Carlo algorithm which constructs a polynomial time (in $p$ and $\log q$) morphism
\[
\Phi: G \to X.
\]
The running time of the algorithm is polynomial in the characteristic $p$ of the underlying field,  in the Lie rank $n$ of $X$,  and and in $\log q$.
\end{theorem}

\subsection{Construction of $C_G(i)$ in black box groups}

Our algorithms are based on the construction of involutions and their centralizers in black box groups. In this subsection we summarize these constructions following \cite{borovik02.7}, see also \cite{bray00.241}.

Let $X$ be a black box group having an exponent $E=2^km$ with $m$ odd. To produce an involution from a random element in $X$, we need an element $x$ of even order. Then the last non-identity element in the sequence
$$1 \neq x^{m}, \, x^{m2}, \, x^{m2^2}, \, \ldots , x^{m2^{k-1}}, x^{m2^k}=1$$
is an involution and denoted by ${\rm i}(x)$. Note that the proportion of elements of even order in classical groups of odd characteristic is at least $1/4$ \cite{isaacs95.139}.

Let $i$ be an involution in $X$. Then, by \cite[Section 6]{borovik02.7}, there is a partial map $\zeta = \zeta_0 \sqcup \zeta_1$ defined by
\begin{eqnarray*}
\zeta: X & \longrightarrow &  C_X(i)\\
x & \mapsto & \left\{ \begin{array}{ll}
\zeta_1(x) = (ii^x)^{(m+1)/2}\cdot x^{-1} & \hbox{ if } o(ii^x) \hbox{ is odd}\\
\zeta_0(x) = {\rm i}(ii^x)  &  \hbox{ if } o(ii^x) \hbox{ is even.}
\end{array}\right.
\end{eqnarray*}

Here $o(x)$ is the order of the element $x\in X$. Notice that, with a given exponent $E$, we can construct $\zeta_0(x)$ and $\zeta_1(x)$ without knowing the exact order of $ii^x$.

The following theorem is the main tool in the construction of centralizers of involutions in black-box groups.

\begin{theorem} {\rm (\cite{borovik02.7})} \label{dist}
Let $X$ be a finite group and $i \in X$ be an involution. If the elements $x \in X$ are uniformly distributed and
independent in $X$, then
\begin{enumerate}
\item the elements $\zeta_1(x)$ are uniformly distributed and independent in $C_X(i)$ and
\item the elements $\zeta_0(x)$ form a normal subset of involutions in $C_X(i)$.
\end{enumerate}
\end{theorem}


By Theorem \ref{dist}, we shall use the map $\zeta_1$ to produce uniformly distributed random elements in $C_X(i)$. For an arbitrary involution $i \in X$ where $X$ is a finite simple classical group, the proportion of  elements of the form $ii^g$ which have odd order is bounded from below by $c/n$ where $c$ is an absolute constant and $n$ is the Lie rank of $X$ \cite{parker10.885}. For the classical involutions in classical groups, such a proportion is proved to be bounded from below by an absolute constant \cite[Theorem 8.1]{suko02}.


\subsection{GAP code and experiments}

All algorithms described in this paper have been implemented in GAP \cite{gap} and thoroughly tested. It is worth noting that our methodology of  building internal structure of a black box group block-by-block, in an organized and directed way, allows us to write (and write quickly!) clean, transparent, compact GAP codes; not surprisingly, de-bugging is also much easier than in the ``elementwise'' approach.

For the efficiency in practice we want to note that the algorithm in Theorem \ref{cons:psl2B} for the group $\sl_2(7^{30})$ constructs the unipotent, toral and Weyl group elements in around 20 seconds in 2008 model standard MacBook.

\subsection{Notation}

The notation is standard and mostly follows \cite{gorenstein1998}. In particular,

\bi
\item $\ppsl_n(q)$ denotes any group $\sl_n(q)/N$ where $N$ is a subgroup of the center of $\sl_n(q)$, with similar conventions for the remaining classical groups. In particular, $\ppsl_2(q)$ denotes one of the group $\psl_2(q)$ or $\sl_2(q)$;
\item $\frac{1}{(2,n)}\sl_n(q)$ denotes the factor groups of $\sl_n(q)$ by the subgroup of order $(2,n)$ from its center.
\item The groups $\om_{2n+1}(q)$, $q>3$, $n\geqslant 3$ are simple, so we drop (P) in the notation for these groups.
\ei

\section{Backgrounds}

\subsection{The Curtis-Tits Theorem}

The main identification theorem used in the classification of the finite simple groups is  so called the Curtis-Tits Theorem which shows that the essential relations in the Steinberg presentation are the ones involving Lie rank $2$ subgroups corresponding to fundamental roots in $\Pi$, that is, edges and non-edges in the Dynkin diagram.  

\begin{theorem}\label{th:27.3} {\rm (Curtis--Tits, \cite[Theorem 27.3]{gorenstein1994})} Let $G^*$ be a finite group of Lie type. Let $\Sigma$ be the root system of $G^*$ and $X_\alpha$ {\rm (}$\alpha\in\Sigma${\rm )} the corresponding
root subgroups. Let\/ $\Pi$ be a fundamental system in $\Sigma$ and for each $\alpha\in\Pi$ set
$$G^*_\alpha =\langle X_\alpha, X_{-\alpha}\rangle.$$ Assume that\/ $|\Pi|\geqslant 3$.

If now $G$ is any group generated
by subgroups $G_\alpha$ {\rm (}$\alpha\in\Pi${\rm )}, and if there are homomorphisms
\[
\phi_\alpha: G^*_\alpha \longrightarrow G_\alpha
\]
 and
 \[
\phi_{\alpha\beta}: \langle G^*_\alpha, G^*_\beta\rangle  \longrightarrow \langle G_\alpha, G_\beta\rangle
\]
 for all $\alpha,\beta\in\Pi$, which are coherent in the sense that $\phi_{\alpha\beta} = \phi_{\beta\alpha}$
 and
 \[
 \phi_{\alpha\beta}\mid_{G^*_\alpha} = \phi_\alpha
 \]
 for all $\alpha$ and $\beta$ in $\Pi$, then $G/Z(G)$ is a homomorphic image of\/ $G^*/Z(G^*)$.
\end{theorem}

The system of subgroups $\{G_\alpha \mid \alpha\in\Pi\}$ which satisfies conditions of Theorem~\ref{th:27.3}  is usually called a \emph{Curtis-Tits system} of the groups $G$.

\subsection{A Curtis-Tits configuration, the single bonds case}
\label{subsec:CTC-single-bond}

In this series of papers \cite{suko03,suko12E,suko12D,suko12B,suko12C,suko02}, we move beyond identification of a simple black box group to creation of tools for computing within these groups. For that purpose, we will use a modified concept of a Curtis-Tits system, more suitable for pin-pointing the internal structure of the given black box group $X$.

\begin{figure}[htbp]
\begin{center}
\setlength{\unitlength}{.025cm}
\begin{picture}(250,250)(-60,40)
{\huge


\put(41,260){\small $K_0$}
\put(-20,221.5){\line(5,2){65}}
\put(117,221.5){\line(-5,2){65}}
\put(41,240){$\bullet$}
\put(-30,210){$\circ$}
\put(-30.2,190){\small $K_1$}
\put(-17,217.5){\line(1,0){20}}
\put(2,210){$\circ$}
\put(2,190){\small $K_2$}
\put(15,217.5){\line(1,0){10}}
\put(31,210){$\cdots$}
\put(73,217.5){\line(1,0){10}}
\put(82,210){$\circ$}
\put(82,190){\small $K_{n-1}$}
\put(95,217.5){\line(1,0){20}}
\put(114,210){$\circ$}
\put(114,190){\small $K_n$}

\put(-58.2,116.6){$\bullet$}
\put(-58.2,136.6){\small $K_0$}
\put(-57.8,63){$\circ$}
\put(-57.8,43){\small $K_1$}
\put(-47,73.5){\line(1,1){20}}
\put(-27.5,100.5){\line(-1,1){20}}
\put(-30,90){$\circ$}
\put(-33,70){\small $K_2$}
\put(-17,97.5){\line(1,0){21}}
\put(2,90){$\circ$}
\put(2,70){\small $K_{3}$}
\put(15,97.5){\line(1,0){10}}
\put(33,90){$\cdots$}
\put(73,97.5){\line(1,0){11}}
\put(82,90){$\circ$}
\put(75,70){\small $K_{n-3}$}
\put(95,97.5){\line(1,0){21}}
\put(114,90){$\circ$}
\put(110,70){\small $K_{n-2}$}
\put(125,100.5){\line(1,1){20}}
\put(145.5,73.5){\line(-1,1){20}}
\put(142.4,116.6){$\circ$}
\put(142.4,136){\small $K_{n-1}$}
\put(142.5,63){$\circ$}
\put(142.5,43){\small $K_n$}
}
\end{picture}
\caption{Extended Dynkin diagrams (and Curtis-Tits systems) of types $A_n$ and $D_n$}
\label{ean}\label{edn}
\end{center}
\end{figure}

\begin{figure}

\setlength{\unitlength}{.025cm}
\begin{center}
\begin{picture}(250,250)(0,0)
{\huge

\put(2,230){$\circ$}
\put(15,237.5){\line(1,0){20}}
\put(34,230){$\circ$}
\put(47,237.5){\line(1,0){20}}
\put(66,230){$\circ$}
\put(79,237.5){\line(1,0){20}}
\put(98,230){$\circ$}
\put(111,237.5){\line(1,0){20}}
\put(130,230){$\circ$}
\put(73,232){\line(0,-1){20}}
\put(66,200){$\circ$}
\put(73,202){\line(0,-1){20}}
\put(66,170){$\bullet$}

\put(2,130){$\bullet$}
\put(12,137.5){\line(1,0){20}}
\put(32,130){$\circ$}
\put(45,137.5){\line(1,0){20}}
\put(64,130){$\circ$}
\put(77,137.5){\line(1,0){20}}
\put(96,130){$\circ$}
\put(109,137.5){\line(1,0){20}}
\put(128,130){$\circ$}
\put(141,137.5){\line(1,0){20}}
\put(159,130){$\circ$}
\put(171,137.5){\line(1,0){20}}
\put(189,130){$\circ$}
\put(103,132){\line(0,-1){20}}
\put(96,100){$\circ$}

\put(2,50){$\circ$}
\put(15,57.5){\line(1,0){20}}
\put(33,50){$\circ$}
\put(45,57.5){\line(1,0){20}}
\put(63,50){$\circ$}
\put(75,57.5){\line(1,0){20}}
\put(93,50){$\circ$}
\put(105,57.5){\line(1,0){20}}
\put(123,50){$\circ$}
\put(135,57.5){\line(1,0){20}}
\put(153,50){$\circ$}
\put(165,57.5){\line(1,0){20}}
\put(183,50){$\circ$}
\put(195,57.5){\line(1,0){20}}
\put(213,50){$\bullet$}
\put(70,52){\line(0,-1){20}}
\put(63,20){$\circ$}
}
\end{picture}
\end{center}

\caption{Extended Dynkin diagrams (and Curtis-Tits systems) of types $E_6$, $E_7$ and $E_8$}
\label{fig:en}

\end{figure}
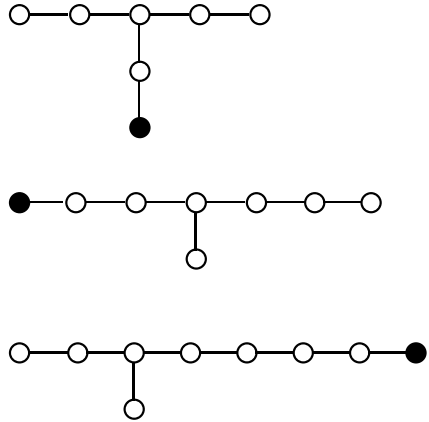

Let $\Phi$ be an irreducible root system of rank at least $3$ with fundamental system $\Pi$ and with Dynkin diagram $\Delta$ of one of the types $A_n$ for $n \geqslant 2$, $D_n$ for $n\geqslant 3$, $E_6$, $E_7$, or $E_8$. Let\/ $\Pi^*=\Pi \cup \{-\alpha_0 \}$ where $\alpha_0$ is the highest root in $\Pi$ and $\Delta^*$ be the extended Dynkin diagram for $\Pi^*$.

Given a black box group $X$ of type $\Delta$ over a finite field of odd prime power order $p^k=q>3$, we have constructed in $X$ \cite{suko03} a \emph{Curtis-Tits configuration}, that is, a system of subgroups $K_\alpha\cong \sl_2(q)$ labeled by nodes of  $\Delta^*$  and satisfying the following conditions:

\begin{itemize}
\item[(a)] $X =\langle K_\alpha \mid \alpha \in \Delta\rangle$;
\item[(b)] $[K_\alpha, K_\beta]=1$ if $\alpha$ and $\beta$ are not connected in $\Delta^*$;
\item[(c)] $\langle K_\alpha,K_\beta \rangle \cong\sl_3(q)$ if $\alpha$ and $\beta$ are connected with a single bond;

\item[(d)] if $z_\alpha$ is the involution in the center of $K_\alpha$ then $[z_\alpha,z_\beta] =1$ for all $\alpha,\beta \in \Delta^*$;
\item[(e)] if we form the elementary abelian group $E = \langle z_\alpha \mid \alpha\in \Delta^*\rangle$, then $H=C_X(E)$ is
    an abelian $p'$-group;
\item[(f)] for all $\alpha\in \Delta^*$, the group $H_\alpha = H\cap K_\alpha$ is a maximal split torus in $K_\alpha$ and $\langle H_\alpha \mid \alpha\in \Delta\rangle = H$;
\item[(g)] the subgroup $H$ normalizes each subgroup $K_\alpha$ for $\alpha\in \Delta^*$.

\end{itemize}

\begin{definition}
The indexed set of subgroups $\{K_\alpha \mid \alpha \in \Delta \}$ which satisfies the conditions {\rm (a)--(g)} above is called a {\rm (}\emph{single bond}{\rm )} \emph{Curtis-Tits configuration}. Similarly,  $\{K_\alpha \mid \alpha \in \Delta^* \}$ is called a {\rm (}\emph{single bond}{\rm )} \emph{extended Curtis-Tits configuration} for\/ $X$.
\end{definition}

\begin{example}
 The following $n-1$ subgroups $\sl_2(q)$ form a Curtis-Tits configuration in $G= \sl_n(q)$.

\setlength{\unitlength}{.025cm}
\begin{picture}(250,290)(-40,0)
\put(70,70){\line(1,0){5}}
\put(70,270){\line(1,0){5}}
\put(70,270){\line(0,-1){200}}
\put(270,270){\line(0,-1){200}}
\put(265,70){\line(1,0){5}}
\put(265,270){\line(1,0){5}}

\put(230,70){{\line(1,0){5}}}
\put(230,110){{\line(1,0){5}}}
\put(270,110){{\line(0,-1){40}}}
\put(230,110){{\line(0,-1){40}}}
\put(265,70){{\line(1,0){5}}}
\put(265,110){{\line(1,0){5}}}
\put(255, 73){*}
\put(235, 73){*}
\put(255,94){*}

\put(210,90){{\line(1,0){5}}}
\put(210,130){{\line(1,0){5}}}
\put(210,130){{\line(0,-1){40}}}
\put(250,130){{\line(0,-1){40}}}
\put(245,90){{\line(1,0){5}}}
\put(245,130){{\line(1,0){5}}}
\put(235, 113){*}
\put(215,94){*}
\put(235,94){*}

\put(166,163){{$\cdot$}}
\put(155,175){{$\cdot$}}
\put(144,187){{$\cdot$}}

\put(190,110){{\line(1,0){5}}}
\put(190,150){{\line(1,0){5}}}
\put(190,150){{\line(0,-1){40}}}
\put(228,150){{\line(0,-1){40}}}
\put(223,110){{\line(1,0){5}}}
\put(223,150){{\line(1,0){5}}}
\put(215, 133){*}
\put(195, 133){*}
\put(195,113){*}
\put(215,113){*}

\put(90,250){{\line(1,0){5}}}
\put(90,210){{\line(1,0){5}}}
\put(90,250){{\line(0,-1){40}}}
\put(130,250){{\line(0,-1){40}}}
\put(125,250){{\line(1,0){5}}}
\put(125,210){{\line(1,0){5}}}
\put(115, 233){*}
\put(95,213){*}
\put(115,213){*}

\put(70,270){{\line(1,0){5}}}
\put(70,230){{\line(1,0){5}}}
\put(70,270){{\line(0,-1){40}}}
\put(110,270){{\line(0,-1){40}}}
\put(105,270){{\line(1,0){5}}}
\put(105,230){{\line(1,0){5}}}
\put(75, 252){*}
\put(95, 252){*}
\put(75,233){*}
\put(95,233){*}


\put(190,210){\Huge $0$}
\put(110,110){\Huge $0$}
\end{picture}
\end{example}

The following result is an immediate consequence of \cite[Corollary 27.4]{gorenstein1994}, which is, in its turn, a special case of Theorem~\ref{th:27.3}.

\begin{theorem} \label{th:single-bond} Let $G$ be a finite group with a single bond Curtis-Tits configuration $\{G_\alpha\}$ over a field of odd order $q>5$ corresponding to one of the Dynkin diagrams $A_n$, $n\geqslant 2$, $D_n$, $n \geqslant 3$, $E_6$, $E_7$, or $E_8$. Then $G$ is isomorphic to a quasi-simple group of Lie type over $\GF(q)$ of the corresponding type and\/ $\{G_\alpha\}$ is the system of root $\sl_2$-subgroups for roots in a system of simple roots $\Pi$ associated with some maximal split torus $H$ of\/ $G$.
\end{theorem}

\subsection{A Curtis-Tits configuration for double bond Dynkin diagrams of rank at least 3}
\label{subsec:CTC-double-bond}

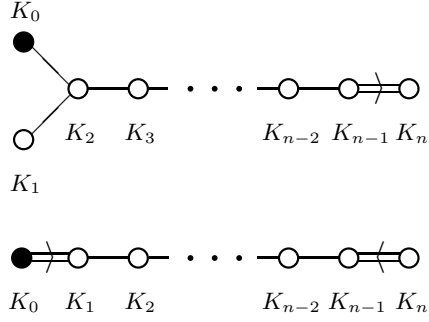
\begin{figure}[htbp]
\begin{center}

\setlength{\unitlength}{.025cm}
\begin{picture}(250,220)(-60,20)
{\huge

\put(-59,185){$\bullet$}
\put(-59,205){\small $K_0$}
\put(-59,133){$\circ$}
\put(-59,113){\small $K_1$}
\put(-47.2,143.5){\line(1,1){20}}
\put(-27.2,169.5){\line(-1,1){20}}
\put(-30,160){$\circ$}
\put(-30,140){\small $K_2$}
\put(-17.2,167.5){\line(1,0){21}}
\put(2,160){$\circ$}
\put(2,140){\small $K_3$}
\put(14.6,167.5){\line(1,0){10}}
\put(33,160){$\cdots$}
\put(73,167.5){\line(1,0){11}}
\put(82,160){$\circ$}
\put(75,140){\small $K_{n-2}$}
\put(94,167.5){\line(1,0){22}}
\put(114,160){$\circ$}
\put(112,140){\small $K_{n-1}$}
\put(126,165.5){\line(1,0){22.3}}
\put(134.5,163.5){\large $\rangle$}
\put(126,169.5){\line(1,0){22}}
\put(146,160){$\circ$}
\put(146,140){\small $K_n$}

\put(-60,70){$\bullet$}
\put(-60,50){\small $K_0$}
\put(-48,75.5){\line(1,0){20}}
\put(-40.5,73.5){\large $\rangle$}
\put(-48,79.5){\line(1,0){20}}
\put(-30,70){$\circ$}
\put(-30,50){\small $K_1$}
\put(-17.2,77.5){\line(1,0){22}}
\put(2,70){$\circ$}
\put(2,50){\small $K_2$}
\put(15,77.5){\line(1,0){10}}
\put(33,70){$\cdots$}
\put(73,77.5){\line(1,0){11}}
\put(82,70){$\circ$}
\put(75,50){\small $K_{n-2}$}
\put(94,77.5){\line(1,0){22}}
\put(114,70){$\circ$}
\put(110,50){\small $K_{n-1}$}
\put(126,75.5){\line(1,0){23}}
\put(134.5,73.5){\large $\langle$}
\put(126,79.5){\line(1,0){22}}
\put(146,70){$\circ$}
\put(146,50){\small $K_n$}
}
\end{picture}
\caption{Extended Dynkin diagrams of type $B_n$ and $C_n$}
\label{ebn}\label{ecn}
\label{arrows}
\end{center}
\end{figure}

Let now $\Delta^*$ be an extended Dynkin diagram of one of the types $B_n$ or $C_n$ for $n \geqslant 3$, see Figure~\ref{arrows}; notice that the labelling of the nodes is chosen so that the nodes $1,2,\dots, n$ form the corresponding Dynkin diagram and the nodes $n-1$ and $n$ are connected by a double bond.

Given a black box group $X$ of type $B_n$ or $C_n$ over a finite field of odd prime power order $p^k=q>5$, \emph{Curtis-Tits configuration}, that is, a system of subgroups $K_\alpha\cong \ppsl_2(q)$ labeled by nodes of  $\Delta^*$  and satisfying the following conditions:


\begin{itemize}
\item[(a)] $X =\langle K_\alpha \mid \alpha \in \Delta\rangle$;
\item[(b)] $[K_\alpha, K_\beta]=1$ if $\alpha$ and $\beta$ are not connected in $\Delta^*$;
\item[(c)]
\begin{itemize}
\item $\langle K_\alpha,K_\beta \rangle \cong\sl_3(q)$ if $\alpha$ and $\beta$ are connected with a single bond;
\item $\langle K_\alpha,K_\beta \rangle \cong\ppsp_4(q)$ if $\alpha$ and $\beta$ are connected with a double bond;
\end{itemize}
\item[(d)]  if $z_\alpha$ and $z_\beta$ are involution in the centers of $K_\alpha \cong K_\beta \cong \sl_2(q)$ then $[z_\alpha,z_\beta] =1$. Moreover, there exits an element $t\in X$ of order $(q-1)/2$ such that $\langle K_1, \ldots, K_{n-1} \rangle \leq C_X(t)$ and $[z_0,t]=[z_n,t]=1$.
\item[(e)] if we form the abelian group \[E = \langle z_\alpha,t \mid \alpha\in \Delta^* \mbox{ and } K_\alpha\cong \sl_2(q) \rangle,\] then $H=C_X(E)$ is an abelian $p'$-group.
\item[(f)] for all $\alpha\in \Delta^*$, the group $H_\alpha = H\cap K_\alpha$ is a maximal split torus in $K_\alpha$ and $\langle H_\alpha \mid \alpha\in \Delta\rangle = H$;
\item[(g)]  the subgroup $H$ normalizes each subgroup $K_\alpha$ for $\alpha\in \Delta^*$.
\end{itemize}

\begin{remark}\label{rem:bncn}
{\rm From the algorithmic point of view, we shall note here that if $X\cong \ppsp_{2n}(q)$ and $\qpone$, then the element $t\in X$ of order $(q-1)/2$ in (d) above can be chosen to be an involution (or a pseudo-involution -- an element of order 4 whose square is a central involution when $X$ is not simple). The restriction $q>5$ is imposed due to the following fact for the groups of type $B_n$. Note first that if $X\cong \om_{2n+1}(q)$, then $C=C_X(z_0, z_1, \ldots, z_{n-1})$ is a subgroup of order $\frac{(q-1)^n}{2}.2^{n}$,  which is not abelian. Here, the subgroup of order $2^n$ arises from the graph automorphisms and the inversions in the corresponding centralizers of involutions. Therefore, if $q=5$, then the element $t$ becomes an involution since its order is $(5-1)/2=2$ and it centralizes the subgroup of order $2^n$. We shall also note that such an element $t$ corresponds to an involution of type $t_n$, that is, $C_X(t)$ has a semisimple component isomorphic to $\om_{2n}^+(q)$, and if $q>5$, then $C_X(t)$ has a semisimple component isomorphic to $\sl_n(q)$. Thus, if $q=5$ then  $C=C_X(z_0, z_1, \ldots, z_{n-1},t)$ is not a maximal split torus. However, in the case of groups of type $B_n$, the black box group algorithm for the construction of a Curtis-Tits configuration does not depend on such an element $t \in X$. We also note here that the Curtis-Tits configuration above for the groups $\ppsp_{2n}(5)$ also holds. For $q = 3$, however, the methods used in the construction of Curtis-Tits configurations in \cite{suko03} do not work as $\sl_2(3)$ is solvable. Hence it is enough to assume that $q>3$ for the algorithmic applications.  We refer the reader to \cite[Section 8]{suko03} for the details about the construction of Curtis-Tits configurations in black box classical groups.
}
\end{remark}

Similarly to Theorem~\ref{th:single-bond}, the following result is an easy consequence of Theorem~\ref{th:27.3}.

\begin{theorem} \label{th:double-bond} Let $G$ be a finite group with an extended double bond Curtis-Tits configuration $\{G_\alpha\}$ over a field of odd order $q>5$ corresponding to one of the extended Dynkin diagrams $B_n$ or $C_n$ {\rm (}$n \geqslant 3${\rm )}.

Then $G$ is isomorphic to a quasi-simple group of Lie type over $\GF(q)$ of the corresponding type and $\{G_\alpha \mid \alpha \in \Delta\}$ is the system of root $\sl_2$-subgroups for roots in a system of simple roots $\Pi$ associated with some maximal split torus $H$ of $G$.
\end{theorem}

Perhaps it is time to comment on subtle differences between the Curtis-Tits configurations in groups $\ppsp_{2n}$ of type $C_n$ and in groups $\om_{2n+1}$  and  ${\rm Spin}_{2n+1}$ of type $B_n$. Indeed some confusion could be created by the fact that $\psp_{4}(q) \cong \om_{5}(q)$ and $\sp_{4}(q) \cong {\rm Spin}_{5}(q)$. However the root $\ppsl_2$-subgroups $K_{n-1}$ and $K_n$ in $\langle K_{n-1}, K_n\rangle \cong \ppsp_4(q)$ correspond to roots of different length, and therefore correspond to homomorphisms $\sl_2(q) \longrightarrow \ppsp_4(q)$ which are \emph{not} conjugate in ${\rm Aut}\, \ppsp_4(q)$.

In particular,
\begin{itemize}
\item If $K_n \cong \psl_2(q)$ then $\langle K_{n-1}, K_n\rangle \cong \psp_4(q) \cong\om_5(q)$ and $G\cong \om_{2n+1}(q)$;
\item if $K_n \cong \sl_2(q)$ then  $\langle K_{n-1}, K_n\rangle \cong\sp_4(q)\cong {\rm Spin}_5(q)$ and $G\cong \ppsp_{2n}(q)$ or ${\rm Spin}_{2n+1}(q)$; in that case, further information comes from the behaviour of the involution $z_n \in Z(K_n)$:
    \begin{itemize}
    \item if $z_n \in Z(G)$ then $G\cong {\rm Spin}_{2n+1}(q)$;
    \item if $z_n \not\in Z(G)$ then $G\cong \ppsp_{2n}(q)$.
    \end{itemize}
\end{itemize}

Actually the distinctions between the cases of $\psp_{2n}$, ${\rm Sp}_{2n}$, $\om_{2n+1}$ and ${\rm Spin}_{2n+1}$ become clear at early stages of construction of an extended Curtis-Tits configuration \cite{suko02,suko03} and therefore any potential confusion is easily avoidable.

Indeed, before we start constructing Curtis-Tits system for a classical group, we construct a long root $\sl_2(q)$-subgroup $K$ and check whether, for a random element $g \in G$, $\langle K,K^g \rangle$ is
\begin{itemize}
\item $\sl_4(q)$,
\item $\sp_4(q)$,
\item $\su_4(q)$, or
\item $\so^+_8(q)$.
\end{itemize}
This is the case with probability at least $1-O(1/q)$ provided that the rank is big enough if $G$ is $\sl_n(q)$, $\sp_{2n}(q)$, $\su_n(q)$ or an orthogonal group, respectively. This is Theorem 7.1 in \cite{suko03}, and Theorem 7.2 of the same paper presents an algorithm which computes the type of the group (not distinguishing the groups of type $B_n$ and $D_n$). The differences between $\sp_{2n}(q)$ and $\psp_{2n}(q)$, and $\om_{2n+1}(q)$ and ${\rm Spin}_{2n+1}(q)$ is quite clear as one of them has a central involution and the other does not have such an involution.

\section{Steinberg generators of $\ppsl_2(q)$}\label{sec:base}

Let $G \cong \psl_2(q)$ and $\qpone$. In this section we construct, in G, a unipotent element $u$, a Weyl group element $w$ which conjugates $u$ to its opposite and the split torus $T<G$ which normalizes the root subgroup $U$ containing $u$. Note that all of this construction can be done in $\sl_2(q)$ with obvious modifications in the arguments.

\begin{remark}\label{rem:regular}
{\rm
It is well-known that the non-trivial elements of $G\cong \psl_2(q)$  are either semisimple or unipotent. Moreover, any non-trivial semisimple (or unipotent) element belongs to a unique maximal torus (or root subgroup).
}
\end{remark}

\begin{lemma}\label{uni:exist}
Let $G\cong \psl_2(q)$, $\qpone$ and $q=p^k$ for some $k\geqslant 1$. If $i \in G$ is an involution, then there exists $g \in G$ such that $ii^g$ has order $p$.
\end{lemma}

\begin{proof}
Since $\qpone$,  any involution in $G$ belongs to a torus of order $(q-1)/2$.  Assume that $i\in T$ for some torus $T<G$. By Remark \ref{rem:regular}, $T$ is uniquely determined by the involution $i$. Moreover, there are exactly two Borel subgroups $B_1$ and $B_2$ which contain $T$. If  $B_1= T \ltimes U$ and $B_2=T\ltimes V$, then $U$ and $V$ are opposite unipotent subgroups of $G$.
 Now observe that $u^i=u^{-1}$ for any $u\in U$ (or $u \in V$), which implies that the element $ii^u$ has order $p$ and the lemma follows.
\end{proof}

\begin{lemma}\label{uni:prop:sl2}
Let $G\cong \psl_2(q)$, $\qpone$ and $q=p^k$ for some $k\geqslant 1$. For any involution $i\in G$, the probability that $ii^g$ has order $p$ for a random element $g\in G$ is at least $1/q$.
\end{lemma}

\begin{proof}
Normalizers of unipotent subgroups of order $q$ in $G$ are Borel subgroups, and it is well known that any two distinct Borel subgroups intersect over a maximal torus of order $(q-1)/2$. Since $\qpone$, this torus contains an involution, uniquely determined by this involution and normalizes exactly two unipotent subgroups of order $q$---see proof of Lemma~\ref{uni:exist}.  Hence there are only two unipotent subgroups $U,V$ of order $q$ which are normalized by $i$.

 Unipotent subgroups $U$ and $V$ are also normalized by the torus $T$ containing $i$, and are the opposite unipotent subgroups of each other in the sense of the root system associated with the torus $T$. Since $C_G(U)=U$, all involutions of the form $i^g$, $g \in G$, which invert $U$ lie in the coset $Ui$. By considering the opposite unipotent subgroup $V$, we have $2(q-1)$ involutions of the form $i^g$, $g \in G$, such that $ii^g$ is an element of order $p$. Since $|G| = q(q^2-1)/2$, the proportion of the elements $g\in G$ such that $ii^g$ is of order $p$ is
\[
\frac{2(q-1)}{q(q^2-1)/2} \cdot (q-1) = \frac{4(q-1)}{q(q+1)} > \frac{1}{q}.
\]
\end{proof}

\begin{lemma}\label{uni:sl2}
Let $G\cong \psl_2(q)$ and $\qpone$. Assume that $i\in G$ is an involution and $u=ii^g$ is a unipotent element for some $g \in G$. If $i \in T$ for some torus $T$ and $U$ is a root subgroup containing $u$, then  $T < N_G(U)$.
\end{lemma}

\begin{proof}
Since $u^i=u^{-1}$, we have $i \in N_G(\langle u \rangle)$. Moreover, since $N_G(U)$ contains a torus and $i \in T$, by Remark \ref{rem:regular}, we have $T< N_G(U)$ and the lemma follows.
\end{proof}

\begin{lemma}\label{uni:gen:sl2}
Let $G\cong \psl_2(q)$ and $U<G$ be the root subgroup containing a unipotent element $u \in G$. Assume also that $T=\langle t \rangle$ is a maximal torus in $N_G(U)$. Then $\langle u, t \rangle = N_G(U)$. In particular, $U$ is the derived subgroup of $\langle u, t \rangle$.
\end{lemma}

\begin{proof}
It is well-known that $U$ is a minimal normal subgroup of $N_G(U)=TU$. Hence the lemma follows.
\end{proof}

\begin{remark}\label{rem:weyl}
{\rm
Let $G\cong \psl_2(q)$ and $\qpone$. Let $T$ be a torus of order $(q-1)/2$ containing an involution $i$. Then $C_G(i)=T\rtimes \langle j \rangle$ where $t^j=t^{-1}$ for all $t \in T$. Observe that if $U$ is a root subgroup normalized by $T$, then $U^j$ is the opposite root subgroup of $U$ in $G$, that is, $G=\langle U,U^j \rangle$, see \cite[Lemma 6.1.1 and 7.2.1]{carter1972}. We will call the element $j$ a \emph{Weyl group element}.
}
\end{remark}

It is well-known that any two maximal tori of fixed order in $\psl_2(q)$ are conjugate. The following lemma will be used to find a conjugating element for a given two tori of order $(q-1)/2$ in $\psl_2(q)$.

\begin{lemma}\label{lem:align}
Let $G\cong \psl_2(q)$ and $\qpone$. Assume that $T_1,T_2$ be two tori of order $(q-1)/2$ in $G$, and $i_1,i_2$ are the involutions in $T_1,T_2$, respectively. Assume also that $|i_1i_2|=m$ is odd. Then $T_1^z=T_2$ where $z=(i_1i_2)^{(m+1)/2}$.
\end{lemma}

\begin{proof}
Assume first that $q>5$. Let $D=\langle i_1,i_2\rangle$, then $D$ is a dihedral group of order $2m$ and $i_1^z=i_2$.  Since $C_G(i_1)=T_1\rtimes \langle j_1 \rangle$ and $C_G(i_2)=T_2\rtimes \langle j_2 \rangle$ where $j_1$ and $j_2$ are involutions inverting $T_1$ and $T_2$, respectively, we have
\[
T_2\rtimes \langle j_2 \rangle=C_G(i_2)=C_G(i_1^z)=C_G(i_1)^z=T_1^z\rtimes \langle j_1 \rangle^z.
\]
Since there is only one cyclic group of order $(q-1)/2$ in $T_2 \rtimes \langle j_2 \rangle$, we have $T_2=T_1^z$. If $q=5$, then $T_1=\langle i_1 \rangle$ and $T_2 = \langle i_2 \rangle$. Therefore, $T_1^z=T_2$ and the lemma follows.
\end{proof}





\section{An algorithm for $\psl_2(q)$}

Let $G\cong \psl_2(q)$ and  $t\in G$ be an element of order $(p\pm 1)/2$ where $(p\pm 1)/2$ is even. Let $s \in \langle t \rangle$ be an involution, $r \in G$ an involution which inverts $t$, and $x\in G$  an element of order $3$ which normalizes $\langle s, r\rangle$. The subgroup $L = \langle s,r,x \rangle \cong \Alt_4$ plays a crucial role in our algorithm. Observe  that the preimage of the elements $s$ and $r$ in $\widetilde{G}\cong \sl_2(q)$ generate the quaternion group $Q$ of order 8 so the preimage  of $L=\langle s,r,x \rangle$ is a subgroup of $N_{\widetilde{G}}(Q)$.

\begin{lemma}\label{the:killer}
Assume that $G\cong \psl_2(q)$, $q=p^k$ for some $k\geqslant 2$ and $t\in G$ is an element of order $(p\pm 1)/2$ where $(p\pm 1)/2$ is even. Let $s \in \langle t \rangle$ be an involution, $r \in G$ an involution which inverts $t$, and $x\in G$  an element of order $3$ which normalizes $\langle s, r\rangle$. Then, except for $p=5,7$, we have $\langle t, x \rangle \cong \psl_2(p)$. Moreover, if  $a$ divides $k$ and $t$ is of order $(p^a\pm1)/2$ where $(p^a\pm1)/2$ is even, then $\langle t, x \rangle \cong \psl_2(p^a)$.
\end{lemma}

\begin{proof}
Let $L=\langle s,r,x \rangle \cong \Alt_4\cong \psl_2(3)$. Observe that $L$ is a subgroup of some $H\leqslant G$ where $H\cong \psl_2(p)$. Since $s=t^m$ for some $m\geqslant 1$, $t \in C_G(s)$ and $t$ is contained in a torus $T$ of order $(q\pm1)/2$ in $C_G(s)$. Now $T\cap H$ has order $(p\pm 1)/2$. Since $T$ is cyclic, it has only one subgroup of order $(p\pm 1)/2$ so $t\in H$. Thus $\langle t,x\rangle \leqslant H$. By the subgroup structure of $\psl_2(p)$, the subgroup $L$ is either a maximal subgroup of $H$ or it is contained in $\Sym_4\leqslant H$. Hence, if $|t|\geqslant 5$, or equivalently, $p\geqslant 9$ then we have $\langle t,x\rangle =H$ since $L$ does not contain elements of order bigger than 5. As we noted above, if $p=3$, then $L\cong \Alt_4\cong \psl_2(3)$.

Observe that if $a$ divides $k$ and $|t|=(p^a\pm 1)/2$, then $t$ belongs to a subgroup $H\cong \psl_2(p^a)$. Assuming that $(p^a\pm1)/2$ is even, the lemma follows from the arguments above.
\end{proof}

\begin{remark}\label{rem:the:killer}
{\rm
Following the notation of Lemma \ref{the:killer}, observe that if $p=5$, then $s=t$, and if $p=7$, then $|t|=4$ and  $\langle t,x\rangle =\Sym_4$. Therefore, in these cases, we have $\langle t,x \rangle \ncong \psl_2(p)$. It is clear that if $G\cong \sl_2(q)$, then, by considering the pseudo-involutions (whose squares are the central involution in $G$), the same result in Lemma \ref{the:killer} holds. Note that,  in this case, we consider the elements $t \in G$ of order $p\pm 1$ where $(p\pm 1)/2$ is even. Similarly, following the notation in Lemma \ref{the:killer}, if $p=5$ or $7$, then $\langle t,x \rangle \ncong \sl_2(p)$.
}
\end{remark}

The following lemma is concerned with the construction of the element of order 3 in $\Alt_4 \leqslant \psl_2(q)$. Let $V = \{1,i_1,i_2,i_3 \}$ be a subgroup of $G$ isomorphic to Klein 4-group. For any random element $g \in G$, denote $j_\ell=i_{\ell}^{g}$ for  $\ell=1,2,3$.

\begin{lemma}\label{elt:3}
Together with the setting above, assume that $t_1=i_1j_2$ has odd order $m_1$. Set $u_1=t_1^{\frac{m_1+1}{2}}$ and $k=i_3^{gu_1^{-1}}$. Assume also that $t_2=i_2k$ has odd order $m_2$ and $u_2=t_2^{\frac{m_2+1}{2}}$. Then  the element $x=gu_1^{-1}u_2^{-1}$ permutes $i_1, i_2, i_3$. In particular, $x \in N_G(V)\leqslant \Sym_4$ and $x$ has order 3.
\end{lemma}

\begin{proof}
Observe first that $i_1^{u_1}=j_2$ and $i_2^{u_2}=k$. Then, since $k=i_3^{gu_1^{-1}}$, we have $i_2^{u_2}=i_3^{gu_1^{-1}}$. Hence $i_2=i_3^{gu_1^{-1}u_2^{-1}}=i_3^x$. Now, we prove that $i_2^x=i_1$. Since $i_2^g=j_2$ and  $j_2^{u_1^{-1}}=i_1$, we have $i_2^x=i_2^{gu_1^{-1}u_2^{-1}}=i_1^{u_2^{-1}}$. We claim that $t_2 \in C_G(i_1)$, which implies that $u_2 \in C_G(i_1)$, so $i_2^x=i_1^{u_2^{-1}}=i_1$. Now, since $i_2 \in C_G(i_1)$, $t_2=i_2k\in C_G(i_1)$ if and only if $k =i_3^{gu_1^{-1}}\in C_G(i_1)$. Recall that $i_1^{u_1}=j_2$. Therefore $k \in C_G(i_1)$ if and only if $i_3^g \in C_G(j_2)=C_G(i_2^g)$, equivalently, $i_3 \in C_G(i_2)$ and the claim follows. It is now clear that $i_1^x=i_3$ since $i_1i_2=i_3$.
\end{proof}

\begin{lemma}\label{elt:3:prob}
Let $t_1$ and $t_2$ be as in Lemma {\rm \ref{elt:3}}. Then the probability that $t_1$ and $t_2$ have odd orders is bounded from below by $1/2-1/2q$.
\end{lemma}

\begin{proof}
Notice that all involutions in $G\cong \psl_2(q)$ are conjugate. Therefore the probability that $t_1$ and $t_2$ have odd orders is the same as the probability of the product of two random involutions from $G$ to be of odd order.

We denote by $a$ one of these numbers $(q\pm 1)/2$ which is odd and by $b$ the other one. Then $|G|=q(q^2-1)/2=2abq$ and $|C_G(i)|=2b$ for any involution $i\in G$. Hence the total number of involutions is
\[
\frac{|G|}{|C_G(i)|}=\frac{2abq}{2b}=aq.
\]

Now we shall compute the number of pairs of involutions $(i,j)$ such that their product $ij$ belongs to a torus of order $a$. Let $T$ be a torus of order $a$. Then $N_G(T)$ is a dihedral group of order $2a$. Therefore the involutions in $N_G(T)$ form the coset $N_G(T)\backslash T$ since $a$ is odd. Hence, for every torus of order $a$, we have $a^2$ pairs of involutions whose product belong to $T$. The number of tori of order $a$ is $|G|/|N_G(T)|=2abq/2a=bq$. Hence, there are $bqa^2$ pairs of involutions whose product belong to a torus of order $a$. Thus the desired probability is
\[
\frac{bqa^2}{(aq)^2}=\frac{b}{q}\geqslant \frac{q-1}{2q}=\frac{1}{2}-\frac{1}{2q}.
\]
\end{proof}

\begin{remark}\label{rem:torus:order}
{\rm
An important part of our algorithm is to find a generator of a torus $T$ of order $(q\pm 1)/2$ in $\psl_2(q)$. However, since finding the exact order of an element involves factorization of integers into primes, we consider the elements $t \in T$ where the order of $t$ is divisible by some primitive prime divisor of $(q\pm 1)/2$. On the other hand, by \cite[I.8]{mitrinovic1996}, the proportion of the elements of order $(q\pm 1)/2$ in $T$ is $O(1/\log \log q)$.

A prime number $r$ is said to be a primitive prime divisor of $p^k-1$  if $r$ divides $p^k-1$ but not $p^i-1$ for $1\leqslant i<k$. By \cite{zsigmondy92.265}, there exists a primitive prime divisor of $p^k-1$ except when $(p,k)=(2,6)$, or $k=2$ and $p$ is a Mersenne prime. Observe that the all of the primitive prime divisors of $p^{2k}-1$ divide $p^k+1$. Therefore the primitive prime divisors of $p^k+1$ are defined to be the primitive prime divisors of $p^{2k}-1$. Here, we shall note that the Mersenne primes which are less than 1000 are 3, 7, 31, 127.

For the practical purposes of our algorithms, we shall be dealing with small primes, for example the primes less than 1000. Assume that $\qpone$ and $q=p^k$ for some prime $p$. If $k=1$, then we can assume that the factorization into primes is possible and we can check whether a given element has order $(p\pm 1)/2$. If $k\geqslant 2$ and $q$ is big, then we can not use factorization of integers into primes to find exact orders of elements. In this case, if $\ppone$, then we look for an element $t \in T$ which satisfies
\[t^{p^k-1}=1,\quad t^{\prod_{i=1}^{k-1}(p^i-1)} \neq 1.\]
Moreover, we also need that the element $t^{\frac{p^k-1}{p-1}}$ has order $(p-1)/2$. If $\pmone$, then we look for the elements $t\in T$ satisfying
\[ t^{p^k+1}=1,\quad t^{\prod_{i=1}^{2k-1}(p^i-1)} \neq 1\] and the element $t^{\frac{p^k+1}{p+1}}$ has order $(p+1)/2$. Note that the prime factorization of $p\pm 1$ can be computed in $O(p)$ time. It follows from \cite[Lemma 2.6]{kantor01.168} that there exists a primitive prime divisor of $p^k \pm 1$ which divides the order of $t$.
}
\end{remark}

\begin{proof}[Proof of Theorem \ref{cons:uni}]
Let $G$ be a simple black box group of Lie type of odd characteristic $p$. Assume first that $p\neq 5,7$. If $G \ncong \psl_2(q)$ or $\2G2(q)$, then we construct a long root $\sl_2(q)$-subgroup $L$ of $G$ by \cite[Theorem 1.1]{suko02}. Now by Lemmas \ref{the:killer}, \ref{elt:3}, \ref{elt:3:prob} and Remark \ref{rem:the:killer}, we construct a subgroup $K\leqslant L\cong \sl_2(q)$ where $K\cong \sl_2(p)$. By \cite{guralnick01.169}, the proportion of the unipotent elements in $K$ is $O(1/p)$. Therefore, we can find a unipotent element from randomly chosen $O(p)$ elements in $K$. If $G\cong \/\2G2(q)$, then $C_G(i)'\cong \psl_2(q^2)$ for any involution $i\in G$. Therefore, by the same arguments above, we can construct unipotent elements in $\2G2(q)$ and $\psl_2(q)$.

Assume now that $p=5$ or $7$. If $q=p^k$ and $k$ has small prime divisor $r$. Then, again by the arguments above, we can construct a subgroup isomorphic to $\ppsl_2(p^r)$ and perform random search in this subgroup to construct a unipotent element. In this case the probability of finding a unipotent element is $O(1/p^r)$.
\end{proof}

\begin{algorithm}\label{algorithm:psl2} Let $G$ be a black box group isomorphic to $\psl_2(q)$ where $\qpone$ and $q=p^k$.
\begin{itemize}
\item[Input:]
\begin{itemize}
\item[$\bullet$] A set of generators of $G$.
\item[$\bullet$] The characteristic $p$ of the underlying field.
\item[$\bullet$] An exponent $E$ for $G$.
\end{itemize}
\item[Output:]
\begin{itemize}
\item[$\bullet$] A root subgroup $U$;
\item[$\bullet$] The maximal torus $T$ normalizing $U$;
\item[$\bullet$] A Weyl group element $w$ where $U^w$ is the opposite root subgroup of $U$.
\end{itemize}
\end{itemize}
\end{algorithm}

Outline of Algorithm \ref{algorithm:psl2} (a more detailed description follows below):
\begin{itemize}
\item[1.] Find the size of the field $q$.
\item[2.] Construct a Klein 4-group $V=\langle i,j\rangle$ in $G$ together with the torus $T$ where $i \in T$ and $j$ inverts $T$.
\item[3.] Construct an element of order 3 in $N_G(V)$.
\item[4.] Construct $H\cong \psl_2(p)$ or $\psl_2(p^2)$ if $\ppone$ or $\pmone$, respectively.
\item[5.] Construct a unipotent element $u \in H$ of the form $u = ii^h$ for $h  \in H$ and conclude that the torus $T$ which contains $i$ is a subgroup of $N_G(U)$ where $U$ is the root subgroup containing $u$ and $j$ is the corresponding Weyl group element.
\end{itemize}

Now we give a more detailed description of Algorithm \ref{algorithm:psl2}.

\bd
\item[{\bf Step 1:}] We compute the size $q$ of the underlying field by Algorithm 5.5 in \cite{yalcinkaya07.171}.\\

\item[{\bf Step 2:}] Let $E=2^km$ where $(2,m)=1$. Take an arbitrary element $g\in G$. If the order of $g$ is even, then the last non-identity element in the following sequence is an involution
\[
1\neq g^m, \, g^{2m},\, g^{2^2m}, \ldots, g^{2^km}=1.
\]
Note that the probability of finding an element of even order in the groups of Lie type of odd characteristic is at least $1/4$ by \cite[Corollary 5.3]{isaacs95.139}. Let $i\in G$ be an involution constructed as above. Then, we construct $C_G(i)$ by the method described in \cite{borovik02.7, bray00.241} together with the result in \cite{parker10.885}. We have $C_G(i) = T\rtimes \langle w \rangle$ where $T$ is a torus of order $(q-1)/2$ and $w$ is an involution which inverts $T$. We follow the arguments in Remark \ref{rem:torus:order} to find a toral element $t \in T$ where $|t|$ is divisible by $(p-1)r$ if $\ppone$, or $(p^2-1)r$ if $\pmone$ where $r$ is a primitive prime divisor of $(q-1)$. Note that $t$ has order $(q-1)/2$ with probability at least $O(1/\log \log q)$. Note also that the coset $Tw$ consists of involutions inverting $T$. Hence we can find an involution $j \in C_G(i)$ which inverts $T$ with probability at least $1/2$. Now it is clear by the construction that $V=\langle i,j \rangle$ is a Klein 4-group, $i\in T$ and $j$ inverts $T$.\\

\item[{\bf Step 3:}] Let $i_1=i$, $i_2=j$, $i_3=i*j$. Then we search for an element $g \in G$ such that $t_1:=i_1i_2^g$ has odd order $m_1$ and $t_2:=i_2i_3^{gu_1^{-1}}$ has odd order $m_2$ where $u_1=t_1^{\frac{m_1+1}{2}}$. By Lemma \ref{elt:3:prob}, we can find such element $g \in G$ with probability at least $1/2 -1/2q$. Now, by Lemma \ref{elt:3}, $x:=gu_1^{-1}u_2^{-1}\in N_G(V)$ has order 3, where $u_2=t_2^{\frac{m_2+1}{2}}$.\\

\item[{\bf Step 4:}] Let $T=\langle t \rangle$ be the torus constructed in Step 2 and $x$ the element of order 3 constructed in Step 3. By Lemma \ref{the:killer}, if $\ppone$, then $H=\langle t',x \rangle \cong \psl_2(p)$ where $t' \in T$, $|t'|=(p-1)/2$.  If $\pmone$, then $H=\langle t',x \rangle \cong \psl_2(p^2)$ where $t'\in T$ and $|t'|=(p^2-1)/2$.\\

\item[{\bf Step 5:}] Notice first that $i \in H$. Assume that $H\cong \psl_2(p)$ and $\ppone$. Then, by Lemma \ref{uni:prop:sl2}, we can find an element $g\in H$ with probability at least $1/p$ such that $u=ii^g$ is a unipotent element in $H$. Since $i \in T$ (see Step 2), by Lemma \ref{uni:sl2}, $T<N_G(U)$ where $U$ is the subgroup containing $u$. Note that we can construct the root subgroup $U$ by Lemma \ref{uni:gen:sl2}. Moreover, by Remark \ref{rem:weyl}, the element $j$ constructed in Step 2 is the corresponding Weyl group element.

If $\pmone$, then, in Step 4, we construct $H \cong \psl_2(p^2)$. Following the same arguments above, we construct a unipotent element of the form $u=ii^g$ for some $g \in H$ with probability at least $1/p^2$ and the rest of the construction is the same as above.
\ed

Notice that algorithms described in this and previous section provide a proof of Theorems~\ref{cons:psl2B} and \ref{cons:psl2}.

\section{Construction of a maximal split torus}\label{sec:torus}

Let $G$ be a quasi-simple classical black box group of odd characteristic isomorphic to $\ppsl_{n+1}(q)$, $\ppsp_{2n}(q)$, $\ppom_{2n+1}(q)$ or $\ppom_{2n}^+(q)$. Assume that  $\{K_0, K_1, \ldots,K_n \}$ is an extended Curtis-Tits configuration of $G$.

In this section, for any odd $q>3$, we describe a method constructing the split tori $T_\ell < K_\ell$, $\ell=0,1,\ldots,n$, which all together generate a maximally split torus $$T =\langle T_k \mid k=1,2, \ldots, n \rangle$$ in $G$ normalizing $K_\ell$ for all $\ell \in \{0,1,\ldots,n\}$. We set that $K_0$ is the root $\sl_2(q)$-subgroup of $G$ corresponding to the extra node in the extended Dynkin diagram of $G$.

Note that an extended Curtis-Tits configuration of $G$ can be constructed by using the algorithm in \cite{suko03} except that $G\cong \ppsp_{2n}(q)$ and $\qmone$. Therefore, we assume that $\qpone$ if $G\cong \ppsp_{2n}(q)$.

\subsection{Groups of type $A_n$}\label{subsec:torus:an}
Assume that $G\cong \ppsl_{n+1}(q)$, $q>3$, $n\geqslant 2$. Note that, for each $\ell=0,1,\dots,n$ (see Figure~\ref{ean}), $K_\ell\cong \sl_2(q)$. Assume that $i_\ell \in K_\ell$ is the unique involution for each $\ell=0,1, \dots, n$.

We set
\begin{itemize}
\item $T_0=C_{K_0}(i_1)=C_{K_0}(i_n)$,
\item $T_1=C_{K_1}(i_2)$,
\item $T_\ell = C_{K_\ell}(i_{\ell-1})$, $\ell=2, \ldots,n$.
\end{itemize}

\begin{lemma}\label{torus:order}
We have $|T_\ell|=q-1$ for each $\ell=0,1,2, \ldots,n$.
\end{lemma}

\begin{proof}
Recall that the involutions $i_\ell$, $\ell=0,1,\ldots,n$, mutually commute with each other. Observe that $i_k \in N_G(K_\ell)$ for all $k,\ell=0,1,\ldots,n$ and $i_{\ell-1}$ acts as an involution of type $t_1$ on $K_{\ell}$ for $\ell =2, \ldots,n$. Hence $|T_\ell|=|C_{K_\ell}(i_{\ell-1})|=q-1$ for $\ell=2,3,\ldots,n$. The other cases are analogous.
\end{proof}

\begin{lemma}\label{nor:tor:an}
The subgroup $\langle T_0,T_1, \ldots, T_n\rangle$ is a maximally split torus normalizing $K_\ell$ for each $\ell=0,1,\ldots,n$. In particular, $\langle T_0,T_1, \ldots, T_n\rangle=\langle T_1, \ldots, T_n\rangle$.
\end{lemma}

\begin{proof}
By Lemma \ref{torus:order}, we have $|T_\ell|=q-1$ for each $\ell=0,1, \ldots, n$. We shall prove that $T_\ell$'s are mutually commuting with each other. We prove that $[T_1,T_2]=1$ and the other cases are treated similarly. Consider $L=\langle K_1,K_2\rangle \cong \sl_3(q)$. Then, by \cite[Theorem 4.5.5 (c)]{gorenstein1998}, $C_L(i_2)=N_L(K_2)$. Therefore $T_1=C_{K_1}(i_2)$ normalizes $K_2$. Thus $C_{K_2}(T_1)$ is a torus in $K_2$. Since $C_{K_2}(i_1)$ is also a torus in $K_2$ we must have $T_2=C_{K_2}(i_1)=C_{K_2}(T_1)$. Thus $[T_1,T_2]=1$. In a similar manner, we have $[T_k,T_\ell]=1$ and  $T_k \leqslant N_G(K_\ell)$ for each $k,\ell=0,1, \ldots, n$ so $T=\langle T_1, \ldots T_n\rangle$ is a maximally split torus normalizing $K_\ell$ for each $\ell=0,1, \ldots, n$. Since $T_0$ commutes with $T_\ell$ for each $\ell=1,\ldots,n$, we have $T_0 \leqslant C_G(\langle T_1, \dots,T_n \rangle)= \langle T_1, \ldots,T_n \rangle$.
\end{proof}

\subsection{Groups of type $C_n$}

Assume that $G\cong \psp_{2n}(q)$, $n\geqslant 2$, $\qpone$. Let  $\Sigma=\{K_0, K_1, \ldots,K_n \}$ be an extended Curtis-Tits configuration of $G$. By Remark \ref{rem:bncn}, we can take an involution $j\in G$ (necessarily of type $t_n$) such that $C_G(j)=LD$ where $L=\langle K_1, \ldots,K_{n-1}\rangle=C_G(j)''\cong \frac{1}{(2,n)}\sl_n(q)$ and $D$ is a dihedral group of order $2(q-1)$. Note that if $G\cong \sp_{2n}(q)$, then $j$ is a pseudo-involution, $L\cong \sl_n(q)$ and $D$ is a torus of order $q-1$.
Note also that if $G\cong \psp_4(q)$, then $K_1\cong \psl_2(q)$, $K_0\cong K_2\cong \sl_2(q)$ and in all the other cases $K_\ell \cong \sl_2(q)$ for each $\ell =0,1,\ldots,n$.

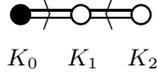
\begin{figure}[htbp]
\begin{center}
\setlength{\unitlength}{.025cm}
\begin{picture}(70,80)(0,30)
{\huge
\put(2,60){$\bullet$}
\put(2,40){\small $K_0$}
\put(14.4,69.5){\line(1,0){22}}
\put(20.5,63.5){\large $\rangle$}
\put(14.4,65.5){\line(1,0){22}}
\put(34,60){$\circ$}
\put(34,40){\small $K_1$}
\put(46,65.5){\line(1,0){22}}
\put(52.5,63.5){\large $\langle$}
\put(46,69.5){\line(1,0){22}}
\put(66,60){$\circ$}
\put(66,40){\small $K_2$}
}
\end{picture}
\caption{Extended Dynkin diagram of $C_2$}
\label{ec2}
\end{center}
\end{figure}

\begin{remark}\label{j:normal}
{\rm
Since $\{K_0, K_1, \ldots, K_n\}$ is a Curtis-Tits configuration of $G\cong \ppsp_{2n}(q)$, the element $j\in G$ chosen above has the property that $j \in N_G(K_i)$ for all $i=0,1,\ldots,n$.
}
\end{remark}

We set
\begin{itemize}
\item $T_0=C_{K_0}(j)$ and $T_n=C_{K_n}(j)$ where $j$ is as above.
\item If $n=2$, then $T_1$ is the cyclic subgroup of order $(q-1)/2$ in $C_{K_1}(i_0)$ where $i_0\in Z(K_0)$.
\item If $n\geq 3$, then $T_1, \ldots,T_{n-1}$ are as described in Subsection \ref{subsec:torus:an}.
\end{itemize}




\begin{lemma}\label{nor:tor:cn}
Let $G\cong \ppsp_{2n}(q)$ and $\qpone$. Then $\langle T_0,T_1,\ldots,  T_2 \rangle$ is a maximally split torus normalizing $K_\ell$ for each $\ell=0,1,\ldots,n$. In particular, $\langle T_0,T_1,\ldots T_n \rangle=\langle T_1, \ldots T_n \rangle$
\end{lemma}

\begin{proof}
We first assume that $G\cong \psp_4(q)$. Then $K_0\cong K_2\cong \sl_2(q)$ and $K_1 \cong \psl_2(q)$, (see Figure \ref{ec2}). We shall show that $T=\langle T_0,T_1,T_2 \rangle$ is abelian. By the setting above, it is clear that $[T_0,T_2]=1$. Moreover, $i_0 \in C_G(j)$ and $i_0\in N_G(K_1)$ where $i_0\in Z(K_0)$ since $\{K_0,K_1,K_2\}$ is a Curtis-Tits configuration of $G$. Now $T_1 \leqslant C_{K_1}(i_0)=C_{K_1}(T_0)=C_{K_1}(T_2)$ since $T_0$ and $T_2$ are cyclic groups and $i_0 \in T_0 \cap T_2$. Hence $[T_1,T_0]=[T_1,T_2]=1$.

We have $T_0=C_{K_0}(j) \leqslant C_G(j)=N_G(K_1)$. Similarly, we have $T_2 \leqslant N_G(K_1)$. Observe that $C_{K_1}(i_0)' \leqslant C_G(i_0)'=N_G(K_0)=N_G(K_2)$ and $|T_1:C_{K_1}(i_0)'|=2$ so $T_1\leqslant N_G(K_0)=N_G(K_2)$. Thus, $T_0$, $T_1$ and $T_2$ normalize $K_0$, $K_1$ and $K_2$. Moreover, since $T_0$ commutes with $T_1$ and $T_2$, we have $T_0\leqslant C_G(\langle T_1,T_2 \rangle)=\langle T_1,T_2\rangle$.

Now we assume that $G\cong \sp_4(q)$ or $\ppsp_{2n}(q)$, $n\geq 3$. We first show that $\langle T_0, T_1, \ldots, T_n \rangle$ is abelian. By Lemma \ref{nor:tor:an}, observe that it is enough to show that $[T_0,T_1]=[T_{n-1},T_n]=1$. We show that $[T_0,T_1]=1$ and the other case is analogous. Consider $L=\langle K_0,K_1 \rangle \cong \sp_4(q)$. Since $\{K_0,K_1, \ldots, K_n\}$ is a Curtis-Tits configuration of $G$,  we have $C_L(i_2)=C_L(i_0)=K_0\times \widetilde{K}_0$ for some $\widetilde{K}_0\cong \sl_2(q)$, and $i_0,i_2$ act as an involution of type $t_1$ on $K_1$. Therefore $C_{K_1}(i_0)=C_{K_1}(i_2)=T_1$. Now since $j$ commutes with $K_1$ and $T_0$, we have $K_1T_0\leqslant C_L(j)$. Therefore $i_0 \in T_0 \leqslant N_L(K_1)$. It follows that $T_1 = C_{K_1}(i_0)=C_{K_1}(T_0)$. Hence $[T_0,T_1]=1$.

It is clear that $T_0$ commutes with $K_\ell$ for each $\ell=2,\ldots,n$. Therefore, since $T_0 \leqslant N_L(K_1)$, we have $T_0 \leqslant N_G(K_\ell)$ for each $\ell=2,\ldots,n$. Moreover, $T_1=C_{K_1}(i_2)=C_{K_1}(i_0)\leqslant C_L(i_0)=N_L(K_0)$. By similar arguments, we conclude that $T_k \leqslant N_G(K_\ell)$ for all $k,\ell =0,1,\ldots, n$.

Since $T_0$ commutes with $T_\ell$ for each $\ell=1,2,\ldots, n$. We have
\[
T_0 \leqslant C_G(\langle T_1, T_2, \ldots,T_n\rangle)=\langle T_1, T_2, \ldots,T_n\rangle.
\]
 Hence $\langle T_1, T_2, \ldots,T_n\rangle=\langle T_0,T_1, T_2, \ldots,T_n\rangle$ and the lemma follows.
\end{proof}

\subsection{Groups of type $B_n$}

Assume that $G\cong \om_{2n+1}(q)$, $q>3$, $n\geqslant 3$. Let $\{K_0, K_1, \ldots,K_n\}$ be an extended Curtis-Tits configuration for $G$ where $K_\ell$, $\ell=0,1,\ldots,n-1$, correspond to long root $\sl_2(q)$-subgroups and $K_n$ corresponds to the short root $\sl_2(q)$-subgroup in the extended Dynkin diagram of $G$ (see Figure~\ref{ecn}). Then $K_\ell \cong \sl_2(q)$ for $\ell=0,1,\ldots,n-1$ and $K_n\cong \psl_2(q)$.
We set
\begin{itemize}
\item[$\bullet$] $T_0=C_{K_0}(i_2)$, $T_1=C_{K_1}(i_2)$,
\item[$\bullet$] $T_\ell=C_{K_\ell}(i_{\ell-1})$, $\ell=2, \ldots, n-1$,
\item[$\bullet$] $T_n < C_{K_n}(i_{n-1})$ where $T_n$ is an abelian group of order $(q-1)/2$.
\end{itemize}

\begin{lemma}\label{nor:tor:bn}
Let $G\cong \om_{2n+1}(q)$, $n\geqslant 3$. Then $\langle T_0,T_1,\ldots,T_n \rangle$ is a maximally split torus normalizing each $K_\ell$, $\ell=0,1,\ldots,n$. In particular, $\langle T_0,T_1,\ldots,T_n \rangle=\langle T_1, \ldots, T_n\rangle$.
\end{lemma}

\begin{proof}
Observe that $\langle K_0,K_2,\ldots,K_{n-1}\rangle \cong \langle K_1,K_2,\ldots,K_{n-1}\rangle \cong \sl_n(q)$ and $\langle K_{n-1}, K_n\rangle \cong \psp_4(q)$. Recall that $K_n\cong \psl_2(q)$ and the involution $i_{n-1} \in K_{n-1}$ acts as an involution of type $t_1$ on $K_n$ so $C_{K_n}(i_{n-1})$ is a dihedral group of order $q-1$. Taking $T_n$ as the abelian group of order $(q-1)/2$ in $C_{K_n}(i_{n-1})$,  the result follows from Lemmas \ref{nor:tor:an} and \ref{nor:tor:cn}.
\end{proof}

\subsection{Groups of type $D_n$}

Assume that $G\cong \ppom_{2n}^+(q)$, $q>3$, $n\geqslant 4$. Let $\{K_0, K_1, \ldots,K_n\}$ be an extended  Curtis-Tits configuration for $G$. 
Then, for each $\ell=0,1,\dots,n$ (see Figure~\ref{ean}), $K_\ell\cong \sl_2(q)$. Assume that $i_\ell \in K_\ell$ is the unique involution for each $\ell=0,1, \dots, n$.

We set
\begin{itemize}
\item[$\bullet$] $T_0=C_{K_0}(i_2)$, $T_1=C_{K_1}(i_2)$,
\item[$\bullet$] $T_{n-1}=C_{K_{n-1}}(i_{n-2})$, $T_n=C_{K_n}(i_{n-2})$,
\item[$\bullet$] $T_\ell=C_{K_\ell}(i_{\ell-1})$, $\ell=2,\ldots,n-2$.
\end{itemize}

\begin{lemma}\label{nor:tor:dn1}
Let $G\cong \ppom_{2n}^+(q)$, $n\geqslant 4$. Then $\langle T_0,T_1,\ldots,T_n \rangle$ is a maximally split torus normalizing each $K_\ell$, $\ell=0,1,\ldots,n$. In particular, $\langle T_0,T_1,\ldots,T_n \rangle=\langle T_1, \ldots, T_n \rangle$.
\end{lemma}
\begin{proof}
It follows from the extended Dynkin diagram of $G$ (see Figure~\ref{ean} on page~\pageref{ean}) that \[\ppsl_n(q) \cong \langle K_0, K_2,K_3, \ldots,K_{n-2}, K_n \rangle \cong \langle K_1, K_2,K_3, \ldots,K_{n-2}, K_{n-1}\rangle.\] Hence the result follows from Lemma \ref{nor:tor:an}.
\end{proof}




\section{Construction of the Weyl group}
In this section, we construct the generators of the Weyl group of a quasi-simple classical group $G$, which correspond to the fundamental reflections in the root  system of $G$.

We assume that $\qpone$ throughout  this section. Let $\{K_0, K_1, \ldots, K_n\}$ be an extended Curtis-Tits configuration of $G$. Assume also that $T_\ell < K_\ell$, $\ell=0,1,\ldots, n$, be the corresponding tori constructed as in Section \ref{sec:torus}. We construct the Weyl group elements $w_\ell \in K_\ell$ as discussed in Remark \ref{rem:weyl} for each $\ell=0,1,\ldots,n$.

\begin{lemma}\label{lem:weyl}
Let $w_\ell \in K_\ell$ be Weyl group elements associated to $T_\ell$, that is, $w_\ell$ inverts $T_\ell$ for each $\ell=0,1,\ldots,n$. Then $w_\ell \in N_G(T)$ for each $\ell=0,1,\ldots,n$ where $T=\langle T_0,T_1, \dots,T_n \rangle$. In particular, $$W=\langle w_0,w_1, \dots,w_n \rangle T/T=\langle w_1, \ldots, w_n\rangle T/T$$ is the Weyl group of $G$.
\end{lemma}

\begin{proof}
We prove that $w_1 \in N_G(T)$ and the other cases are analogous. Assume first that $G\cong \ppsl_n(q)$ and  $L=\langle K_1, K_2 \rangle \cong \sl_3(q)$. Since $w_1$ inverts $T_1$ and $[T_1,T_2]=1$, we have $[T_1^{w_1},T_2]=1$ which implies that $T_2^{w_1} \leqslant C_L(T_1)=\langle T_1,T_2\rangle$. Hence $w_1 \in N_L(\langle T_1,T_2 \rangle)$. Similarly, $w_1 \in N_L(\langle T_1,T_0 \rangle)$. By the construction of $K_0,K_1, \ldots, K_n$, it is clear that $w_1$ commutes with $T_\ell$ for $\ell \geqslant 3$. Thus $w_1 \in N_G(T)$.

If $G\cong \ppsp_{2n}(q)$, then it is enough to consider $L=\langle K_0, K_1 \rangle \cong \ppsp_4(q)$. Then, following the same arguments above, we see that $w_1 \in N_G(T)$. The groups of type $B_n$ and $D_n$ are treated similarly.

Observe that, in all cases, we have $w_0 \in N_G(T)$ by the same arguments in the case $G\cong \psl_n(q)$. Since $N_G(T)=\langle w_1, \ldots , w_n\rangle T$, we have $w_0\in \langle w_1, \ldots, w_n\rangle T$ and the lemma follows.
\end{proof}

\section{Construction of root elements in $\ppsl_{n+1}(q)$}\label{sec:root}

Let $G\cong \ppsl_{n+1}(q)$ and $\{K_0,K_1, \ldots, K_n \}$ be an extended Curtis-Tits configuration for $G$. In this section, we construct unipotent elements $u_0,u_1, \ldots, u_n$ in $K_0,K_1,\ldots,K_n$ where the maximal split torus $T=\langle T_0, T_1, \ldots, T_n \rangle$ constructed in Section \ref{sec:torus} normalizes the root subgroups $U_\ell < K_\ell$ containing $u_\ell$ for each $\ell=0,1,2,\ldots,n$. Note that $T_\ell < K_\ell$ for each $\ell=0,1,2,\ldots,n$.

By Algorithm \ref{algorithm:psl2}, we can construct a triple $(u,T,w)$ such that $u \in K_1\cong \sl_2(q)$ is a unipotent element, $T<K_1$ is a torus of order $q-1$ normalizing the root subgroup containing $u$ and $w\in K_1$ is a Weyl group element. By Lemmas \ref{lem:align} and \ref{elt:3:prob}, we can find an element $g \in K_1 $ such that $T^g=T_1$ with probability at least $1/2$. We set $u_1=u^g$ and $w_1=w^g$. Then, it is clear that $T_1$ is a maximal torus in $K_1$ normalizing the root subgroup $U_1$ containing $u_1$. Moreover $w_1$ inverts $T_1$ and $U_1^w$ is the opposite root subgroup of $U_1$.

For each tori $T_0,T_1, \ldots, T_n$, let $w_0,w_1,\ldots,w_n$ be the corresponding Weyl group elements constructed as discussed in Remark \ref{rem:weyl}, then, by Lemma \ref{lem:weyl}, $$W=\langle w_0,w_1,\ldots,w_n \rangle T/T=\langle w_1,\ldots,w_n \rangle T/T$$ is the Weyl group of $G$.

\begin{lemma}\label{lem:root:subs}
Let $G\cong \ppsl_{n+1}(q)$. If $ \{w_1, \ldots, w_n\}$ is a set of fundamental reflections in the Weyl group $W$ of $G$. Then $$\alpha_i=w_{i-1}w_i(\alpha_{i-1})$$ where $\alpha_i$ is the corresponding fundamental root in the root system of $G$ and $i=2,\ldots,n$. Moreover, $$\alpha_0=w_0w_nw_0(\alpha_n).$$
\end{lemma}

\begin{proof}
The proof follows from a direct computation in the structure of the root system of type $A_n$.
\end{proof}

\begin{corollary}\label{cor:root:subs}
Let $G\cong \ppsl_{n+1}(q)$. We have $$U_i=U_{i-1}^{w_{i-1}w_i}$$ for each $i=2,\ldots,n$ and $$U_0=U_n^{w_0w_nw_0}.$$
\end{corollary}


\section{The algorithm}

\begin{algorithm}\label{the:algorithm} Let $G\cong \ppsl_{n+1}(q)$, $\ppsp_{2n}(q)$, $\om_{2n+1}(q)$, or $\ppom_{2n}^+(q)$ where $q=p^k$ and $\qpone$.
\begin{itemize}
\item[Input: ]
\subitem $\bullet$ a set of generators for $G$;
\subitem $\bullet$ the characteristic $p$ of the underlying field;
\subitem $\bullet$ an exponent for $G$.
\item[Output:]
\subitem $\bullet$ An extended Curtis-Tits system $\{K_0, K_1, \ldots,K_n\}$ for $G$ together with
\subsubitem $\bullet$ The root subgroups $U_\ell < K_\ell$ for each $\ell=0,1,\ldots,n$;
\subsubitem $\bullet$ The maximally split torus $$T=\langle T_0, T_1, \ldots,T_n \rangle$$ where $T_k < N_G(U_\ell)$ for all $k,\ell=0,1,\ldots,n$;
\subsubitem $\bullet$ The Weyl group elements $w_\ell \in K_\ell$, where $U_\ell^{w_\ell}$ is the opposite root subgroup of $U_\ell$ for each $\ell = 0,1,\ldots,n$ and $$\langle w_0,w_1,\ldots,w_n\rangle T/T=\langle w_1, \ldots, w_n\rangle T/T$$ is the Weyl group of $G$.
\end{itemize}
\end{algorithm}

The details of  Algorithm \ref{the:algorithm} are as follows:

\begin{itemize}
\item[1.] Construct an extended Curtis-Tits configuration $\Sigma=\{K_0, K_1, \ldots, K_n\}$ of $G$ and find the size $q$ of the underlying field.
\item[2.] Construct a maximally split torus $T=\langle T_0,T_1,\ldots, T_n\rangle$ and Weyl group elements $w_0,w_1,\ldots,w_n$, where $T_\ell \leqslant K_\ell$   and $T$ normalizes $K_\ell$ for each $\ell=0,1,\ldots,n$.
\item[3.] Construct a subgroup $H_1\cong \sl_2(p)$ or $\sl_2(p^2)$ if $\ppone$ or $\pmone$, respectively, in  $K_1$.
\item[4.] Construct $(u,S,w)$ in $H_1$, where $u$ is a unipotent element, $S$ is a maximal torus normalizing the root subgroup containing $u$ and $w$ is a Weyl group element which inverts $S$.
\item[5.] Construct the maximal torus $T\leqslant K_1$ containing $S$.
\item[6.] Construct $z\in K_1$ such that $T^z=T_1$.
\item[7.] Construct the remaining unipotent elements in each $K_\ell$, $\ell=0,2,3,\ldots,n$.
\end{itemize}

And what follows is a more detailed descriptions of Steps 1--7.

\bd
\item[{\bf Step 1:}] We use the algorithm in \cite{suko03} to construct an extended Curtis-Tits configuration \[\Sigma=\{K_0, K_1, \ldots, K_n\}\] of $G$. We compute the size $q$ of the underlying field by using \cite[Algorithm 5.5]{yalcinkaya07.171}.

\item[{\bf Step 2:}] We construct a maximally split torus $T=\langle T_0,T_1,\ldots,T_n\rangle$ as described in Section \ref{sec:torus} depending on the type of the group $G$. Here the construction of the tori $T_k$, $k=0,1,\ldots,n$, means that we find a toral element $t_k \in T_k$ as in Step 2 of Algorithm \ref{algorithm:psl2}. Moreover, by Remark \ref{rem:weyl}, we can construct the corresponding Weyl group elements $w_0,w_1,\ldots,w_n$ in $K_0, K_1 \ldots, K_n$, respectively. By Lemma \ref{lem:weyl}, $$W=\langle w_0,w_1,\ldots,w_n\rangle T/T$$ is the Weyl group of $G$.

\item[{\bf Step 3:}]
Assume that $K_1\cong \psl_2(q)$. Then this is Step 2, 3 and 4 of Algorithm \ref{algorithm:psl2}. Observe that the same computations apply for $K_1 \cong \sl_2(q)$ with obvious modifications in the arguments as noted in the beginning of Section \ref{sec:base}.

\item[{\bf Step 4:}] This is Step 5 of Algorithm \ref{algorithm:psl2}.

\item[{\bf Step 5:}] We continue to assume that $K_1\cong \psl_2(q)$. If $i\in S$ is an involution, then $C_{K_1}(i)$ contains a torus $T$ of order $(q-1)/2$ containing $S$. It is clear that $T$ normalizes the the root subgroup $U_1$ which contains $u$. We construct the root subgroup $U_1$ by using Lemma \ref{uni:gen:sl2}.

\item[{\bf Step 6:}] Let $T_1$ and $T$ be the tori constructed in Step 2 and Step 5, respectively. Then, by Lemmas \ref{lem:align} and \ref{elt:3:prob}, we can construct an element $z \in K_1$ such that $T^z=T_1$ with probability at least $1/2$.

\item[{\bf Step 7:}] If $G\cong \ppsl_{n+1}(q)$, then, by Corollary \ref{cor:root:subs},  we construct the remaining unipotent elements $u_\ell \in K_\ell$  for each $\ell =0,2,\dots, n$.

If $G\cong \ppsp_{2n}(q)$, then,  since
\[
\langle K_1, \ldots, K_{n-1}\rangle \cong \frac{1}{(2,n)}\sl_n(q) \mbox{ or } \sl_n(q),
\]
we can  construct unipotent elements $u_2\in K_2, \ldots, u_{n-1} \in K_{n-1}$ by Corollary \ref{cor:root:subs}. To construct $u_0$ and $u_n$, we repeat Steps 3, 4, 5, 6 for the groups $K_0$ and $K_n$. It is clear that the unipotent elements $u_0$ and $u_n$ are aligned with the rest of the unipotent elements since $T_0$ and $T_n$ are aligned with the rest of the root subgroups. However, we need to check whether $u_0$ and $u_n$ commute with $u_1$ and $u_{n-1}$, respectively, since $u_0$ or $u_n$ may correspond to opposite root subgroups in $K_0$ and $K_n$, respectively.

If $G\cong \om_{2n+1}(q)$, then the construction of the remaining unipotent elements is similar to the construction of unipotent elements for $\psp_{2n}(q)$.

If $G\cong \ppom_{2n}^+(q)$, then \[\langle K_1,K_2, \ldots, K_{n-1} \rangle \cong \sl_n(q).\] Therefore, we can  construct unipotent elements $u_2\in K_2, \ldots, u_{n-1} \in K_{n-1}$ by Corollary \ref{cor:root:subs}. We follow the same arguments
in the case of $\ppsp_{2n}(q)$ to construct $u_0 \in K_0$ and $u_n\in K_n$.

Finally, by Lemma \ref{uni:gen:sl2}, we construct the root subgroups $U_0, U_1, \ldots, U_n$ for each type of the group $G$.
\ed

Algorithms described in this section provide proof of Theorem~\ref{cons:classical}.

\section*{Acknowledgements}

This paper would have never been written if the authors did not enjoy the warm hospitality offered to them at the Nesin Mathematics Village (in \c{S}irince, Izmir Province, Turkey) in August 2011 and August 2012; our thanks go to Ali Nesin and to all volunteers and staff who have made the Village a mathematical paradise.

We thank Adrien Deloro for many fruitful discussions, in \c{S}irince and elsewhere.

The first author is grateful to Dr Douglas E Jeffrey for most helpful advice.

We gratefully acknowledge the use of Paul Taylor's Commutative Diagrams package, \url{http://www.paultaylor.eu/diagrams/}.

\end{document}